\documentclass[11pt]{amsart}

\usepackage{pb-diagram}

\usepackage{amsfonts}
\usepackage{amsmath}
\usepackage{amssymb}

\newcommand{\ddt}[1]{\frac{\partial #1}{\partial t}}

\newcommand{\Ric}{\textnormal{Ric}}
\newcommand{\Rm}{\textnormal{Rm}}

\newcommand{\cN}{\mathcal{N}}
\newcommand{\vol}{\textnormal{Vol}}
\newcommand{\var}{\textnormal{Var}}

\newtheorem{theorem}{Theorem}[section]
\newtheorem{lemma}[theorem]{Lemma}

\newtheorem{proposition}[theorem]{Proposition}

\numberwithin{equation}{section}

\theoremstyle{definition}

\theoremstyle{definition}
\newtheorem{definition}[theorem]{Definition}

\begin{document}

\title{On the improved no-local-collapsing theorem of Ricci flow}

\author{Wangjian Jian}

\address{Academy of Mathematics and Systems Science$\And$Hua Loo-Keng Key Laboratory of Mathematics, Chinese Academy of Sciences, Beijing, 100190, China.}

\email{wangjian@amss.ac.cn}

\begin{abstract}  In this note we derive an improved no-local-collapsing theorem of Ricci flow under the scalar curvature bound condition along the worldline of the basepoint. It is a refinement of Perelman’s no-local-collapsing theorem.
\end{abstract}

\maketitle


\section{Introduction}

In 1982, Hamilton \cite{Ham} introduced the Ricci flow equation, which is defined by
$$\ddt {g(t)}=-2\Ric(g(t)).$$
Since then, Ricci flow has became a very powerful tool in understanding the geometry and topology of Riemannian manifolds. In Perelman's groundbreaking work \cite{Per1} \cite{Per2} \cite{Per3}, the Ricci flow was used to give the proof of the Poincar\'e and geometrization conjecture.

The no-local-collapsing theorem of Perelman is important in his work of Ricci flow. In his work, whenever we need to take a Gromov-Hausdorff limit, we need to apply the no-local-collapsing theorem. Perelman introduced two approaches to prove the no-local-collapsing theorem: $\mathcal{W}$-entropy and spacetime geometry, and both of them have became very useful tool in the study of Ricci flow theory.

Let us first recall Perelman's no-local-collapsing theorem.

\begin{theorem}[Perelman \cite{Per1}] \label{nlc1}  For any $A<\infty$ and dimension $n$, there exists $\kappa=\kappa(n, A)>0$ such that the following holds:

Let $(M,(g(t))_{t\in [-r^2_0,0]})$ be a Ricci flow on compact, $n$-dimensional manifold that satisfies:
\begin{enumerate}
\item $|\Rm|\leq r_0^{-2}$, on $P(x_0, -r_0^2, r_0; r_0^2)$; 
\item $\vol_{g(-r_0^2)}(B_{g(-r_0^2)}(x_0, r_0))\geq A^{-1}r_0^n$.
\end{enumerate}
Then for any other metric ball $B_{g(0)}(x, r)\subset B_{g(0)}(x_0, Ar_0)$ with $r\leq r_0$ and 
\begin{equation}\label{Equ: Rm bound on final time ball}
|\Rm|\leq r^{-2}, ~on~ P(x, 0, r; -r^2),
\end{equation}
we have
\begin{equation}\label{Equ: nlc1}
\vol_{g(0)}(B_{g(0)}(x, r))\geq \kappa r^n.
\end{equation}
\end{theorem}

Next, let us recall some extensions of this no-local-collapsing theorem.

In \cite{Zh}, Zhang localized Perelman's $\mathcal{W}$-entropy and the differential Harnack inequality, and proved a uniform local Sobolev inequality for the metric ball $B_{g(0)}(x_0, Ar_0)$ (see \cite[Theorem 6.3.2]{Zh}). Then he showed that the curvature condition (\ref{Equ: Rm bound on final time ball}) in Theorem \ref{nlc1} can be relaxed to scalar curvature bound at the final time slice, namely
\begin{equation}\label{Equ: scalar curvature bound on final time ball}
R \leq r^{-2}, ~on~ B_{g(0)}(x, r).
\end{equation}

Then, Wang \cite{W} gave an independent proof of this improved no-local-collapsing theorem, where he can relax the curvature condition (1) in Theorem \ref{nlc1} to the Ricci curvature bound
\begin{equation}\label{Equ: Ricci curvature bound on initial parabolic ball}
|\Ric|\leq r_0^{-2}, ~on~ P(x_0, -r_0^2, r_0; r_0^2),
\end{equation}
and Wang also found its application to K\"ahler-Ricci flow on smooth minimal models of general type.

Lately, under the same curvature condition (\ref{Equ: Ricci curvature bound on initial parabolic ball}), Tian-Zhang \cite{TiZ} established a relative volume comparison theorem of Ricci flow, that is, they remove the non-collapsing condition (2) in Theorem \ref{nlc1} for the initial metric, and obtain
\begin{equation}\label{Equ: relative volume comparison}
\frac{\vol_{g(-r_0^2)}(B_{g(-r_0^2)}(x_0, r_0))}{r_0^n}\geq \kappa \frac{\vol_{g(0)}(B_{g(0)}(x, r))}{r^n}.
\end{equation}
Using this estimate, Song-Tian-Zhang \cite{STZ} obtained the diameter bound for the long time solution of the normalized K\"ahler-Ricci flow under the assumption of a local Ricci curvature bound.

The author and Song \cite{JS} gave an alternative proof of Tian-Zhang's \cite{TiZ} relative volume comparison theorem by using the Nash entropy; using the idea of the proof, we obtained the diameter bound for the long time solution of the normalzied K\"ahler-Ricci flow without assuming local Ricci curvature bound.

The Nash entropy used in \cite{JS} was introduced by Hein-Naber \cite{HN}, where they obtained heat kernel bounds and $\varepsilon$-regularity theorem for Ricci flow. Later in \cite{Bam20a}, Bamler established systematic results on the Nash entropy and heat kernel bounds on Ricci flow background. With these bases, in \cite{Bam20b} \cite{Bam20c}, Bamler established the compactness theory of Ricci flow and the structure theory of non-collapsed limits of Ricci flows. 

Our main result in this note is the following theorem which improves Perelman's no-local-collapsing theorem.

\begin{theorem} \label{nlc2}  For any $A<\infty$ and dimension $n$, there exists $\kappa=\kappa(n,  A)>0$ such that the following holds:

Let $(M,(g(t))_{t\in [-2r^2_0,0]})$ be a Ricci flow on compact, $n$-dimensional manifold, with a base point $x_0$ that satisfies:
\begin{enumerate}
\item $R(x_0, t)\leq \frac{A}{|t|}$, for $t\in[-r_0^2, 0)$; 
\item $\vol_{g(-r_0^2)}(B_{g(-r_0^2)}(x_0, r_0))\geq A^{-1}r_0^n$.
\end{enumerate}
Then for any other metric ball $B_{g(0)}(x, r)\subset B_{g(0)}(x_0, Ar_0)$ with $r\leq r_0$ and 
\begin{equation}\label{Equ: R bound on final time ball}
R\leq r^{-2}, ~on~ B_{g(0)}(x, r),
\end{equation}
we have
\begin{equation}\label{Equ: nlc2}
\vol_{g(0)}(B_{g(0)}(x, r))\geq \kappa r^n.
\end{equation}
\end{theorem}

In the proof of the theorem, we shall also use the Nash entropy and heat kernel bounds on Ricci flow background, which are establisehed by Bamler in \cite{Bam20a}. The ingredient of the proof is that we need to estimate the $W_1$-distance of two points on the same worldline by using the heat kernel upper bound of Ricci flow.

\noindent{\bf Acknowledgements} The author would like to thank Gang Tian, Jian Song, Zhenlei Zhang, Yalong Shi for inspiring discussions.


\section{Preliminaries}
We recall some basic definitions and results of Ricci flow in this section. This is the framework establisehed by Bamler in \cite{Bam20a}.

Let $(M, (g(t))_{t\in I})$ be a smooth Ricci flow on a compact $n$-manifold with $I\subset\mathbb{R}$ being an interval. Let $(x_0, t_0)\in M\times I$ with $A, T^-, T^+ \geq 0$, the parabolic neighborhood is defined as
$$P(x_0, t_0, A; -T^-, T^+):=B(x_0, t_0, A)\times ([t_0-T^-, t_0+T^+]\cap I).$$
We consider the heat operator
$$\Box:=\partial_t -\Delta_{g(t)},$$
and the conjugate heat operator
$$\Box^*:=-\partial_t -\Delta_{g(t)}+R,$$
where $R$ denotes the scalar curvature at time $t$. 

For any $x,y\in M$ and $s,t\in I$ with $s\leq t$, we denote by $K(x,t;y,s)$ the heat kernel along the Ricci flow. That is, for fixed $(y,s)$, the function $K(\cdot,\cdot;y,s)$ is a heat kernel based at $(y,s)$:
$$\Box K(\cdot,\cdot;y,s)=0;~~~~ \lim_{t \searrow s}K(\cdot,t;y,s)=\delta_y,$$
and for fixed $(x,t)$, the function $K(x,t;\cdot,\cdot)$ is a conjugate heat kernel based at $(x,t)$:
$$\Box^* K(x,t;\cdot,\cdot)=0;~~~~ \lim_{s \nearrow s}K(x,t;\cdot,s)=\delta_x.$$

\begin{definition} 
For $(x, t)\in M\times I$ and $s\in I$ with $s\leq t$, we denote by $\nu_{x, t; s}$ the conjugate heat kernel measure based at $(x, t)$ as follows:
$$d\nu_{x, t; s}:=K(x, t; \cdot, \cdot) dg(t)=: (4\pi \tau)^{-n/2} e^{-f} dg(t),$$
where $\tau = t_0 -t$ and $f\in C^{\infty}(M\times (I\cap (\infty, t)))$ is called the potential.
\end{definition}

Basically, in Bamler's viewpoint, instead of considering a point $(x, t)$ in spcaetime, we shall consider its corresponding conjugate heat kernel measure $d\nu_{x, t; s}$ as defined above. Then Bamler used the $W_1$-Wasserstein distance to define distance between the conjugate heat kernel measures based at different points. If $\mu_1$ and $\mu_2$ denotes two probability measures
on a complete manifold $M$ and $g$ is a Riemannian metric on $M$, then we define the $W_1$-distance between $\mu_1$ and $\mu_2$ by
$$d_{W_1}^g(\mu_1, \mu_2):=\sup_{f}\left(\int_{M}fd\mu_1-\int_{M}fd\mu_2\right),$$
where the supremum is taken over all bounded, $1$-Lipschitz function $f:M\to \mathbb{R}$.

Bamler also used the important notion called variance, which in some sense is the $L^2$-norm of the probability measures on a Riemannian manifold. The definition is the following
$$\var (\mu_1, \mu_2):=\int_M\int_Md^2(\mu_1, \mu_2)d\mu_1(x_1)d\mu_2(x_2).$$

If $(M, (g(t))_{t\in I})$ is a Ricci flow and if there is chance of confusion, then we will also write
$\var_t$ for the variance with respect to the metric $g(t)$. We have the following basic relation between the variance and the $W_1$-distance:
\begin{equation}\label{Equ: relation of W1 and variance}
d^g_{W_1}(\mu_1, \mu_2)\leq \sqrt{\var (\mu_1, \mu_2)},
\end{equation}
for any two probability measures.

Now we can define the $H_n$-center of a given base point in Ricci flow.
\begin{definition} 
A point $(z, t)\in M\times I$ is called an $H_n$-center of a point $(x_0, t_0)\in M\times I$ if $t\leq t_0$ and
$$\var_t(\delta_z, \nu_{x_0, t_0; t})\leq H_n(t_0-t).$$
\end{definition}
Given $(x_0, t_0)\in M\times I$ and $t\leq t_0$, there always exists at least one $z\in M$ such that $(z, t)\in M\times I$ is an $H_n$-center of $(x_0, t_0)$, (see \cite[Proposition 3.12]{Bam20a}),and for such $H_n$-center, using inequality (\ref{Equ: relation of W1 and variance}), we always have
$$d^g_{W_1}(\delta_z, \nu_{x_0, t_0; t})\leq \sqrt{\var (\delta_z, \nu_{x_0, t_0; t})}\leq \sqrt{H_n(t_0-t)}.$$

Next, we come to define the so-called Nash entropy which was introduced by Hein-Naber \cite{HN}. Let $(M, g)$ be a closed Riemannian manifold of dimension $n$, $\tau>0$ and $d\nu=(4\pi \tau)^{-n/2} e^{-f} dg$ is a probability measure on $M$, where $f\in C^\infty(M)$, then the Nash entropy is defined by
$$ \cN[g, f, \tau] = \int_M  f d\nu-\frac{n}{2}.$$
Then we can define
\begin{definition} 
Consider a conjugate heat kernel measure $d\nu_{x_0, t_0; t}:=K(x_0, t_0; \cdot, t) dg(t)= (4\pi \tau)^{-n/2} e^{-f_t} dg(t),$ based at some point $(x_0, t_0)\in M\times I$, where $\tau=t_0-t$, the pointed Nash entropy based at $(x_0, t_0)$ is defined as 
$$\cN_{x_0, t_0}(\tau):=\cN[g_{t_0-\tau}, f_{t_0-\tau}, \tau].$$
We set $\cN_{x_0, t_0}(0)=0$. For $s<t_0$, $s\in I$, we will write
$$\cN^*_s(x_0, t_0):=\cN_{x_0, t_0}(t-s).$$
\end{definition}
The pointed Nash entropy $\cN_{x_0, t_0}(\tau)$ is non-increasing when $\tau\geq 0$ is increasing. For more basic properties of the pointed Nash entropy, We refer the readers to \cite[Proposition 5.2]{Bam20a}. (See also Hein-Naber's work \cite{HN}.)

The following lemma is proved in \cite[Theorem 5.9, Corollary 5.11]{Bam20a}.
\begin{lemma} \label{grad} If $R(\cdot, s) \geq R_{min}$ for some $s\in I$, then on $M \times I \cap \left((s, \infty)\right)$,
\begin{equation}\label{Equ: gradient estimate of Nash entropy}
|\nabla \cN^*_s| \leq \left(\frac{n}{2(t-s)} + | R_{min} |\right)^{1/2}.
\end{equation}
Furthermore, if $s< t^* \leq \min{t_1, t_2}$ and $s, t_1, t_2\in I$, then for any $x_1, x_2\in M$,
\begin{equation}\label{Equ: Harnack estimate of Nash entropy}
\begin{split}
&\cN^*_s(x_1, t_1) - \cN^*_s(x_2, t_2)  \\
\leq & \left( \frac{n}{2(t^* -s)} + |R_{min} |\right)^{1/2} d_{W_1}^{g(t^*)}( \nu_{x_1, t_1}(t^*), \nu_{x_2, t_2}(t^*)) + \frac{n}{2} \log \left(\frac{t_2-s}{t^*-s}\right).\\
\end{split}
\end{equation}
\end{lemma}

The following volume non-collapsing estimate is proved in \cite{Bam20a} as a generalization of Perelman's $\kappa$-non-collapsing theorem.

\begin{lemma} \label{noncol} Let $(M, (g(t))_{t\in [-r^2,  0]})$ be a solution of the Ricci flow. If
$$R \leq r^{-2}, ~ on ~ B_{g(0)}(x, r)\times [-r^2, 0], $$
then
\begin{equation}\label{Equ: volume lower bound by Nash entropy}
\vol_{g(0)} (B_{g(0)}(x, r)) \geq c \exp (\cN^*_{-r^2}(x, 0))r^n.
\end{equation}
\end{lemma}

The following volume non-inflation estimate is also proved in \cite{Bam20a}.

\begin{lemma} \label{noninfl} Let $(M, (g(t))_{t\in [-r^2,  0]})$ be a solution of the Ricci flow. If
$$R\geq -r^{-2}    $$
on $M\times [-r^2, 0]$, then for any $A\geq 1$, there exists $C_0=C_0(n), C=C(n)>0$ such that
\begin{equation}\label{Equ: volume upper bound by Nash entropy}
\vol_{g(0)}(x, Ar) \leq C \exp(\cN^*_{-r^2}(x, 0)+C_0A^2)r^n.
\end{equation}
\end{lemma}

The following heat kernel upper bound estimate is proved in \cite{Bam20a} (Theorem 7.2).

\begin{lemma} \label{hkub} Let $(M, (g(t))_{t\in I})$ be a solution of the Ricci flow. Suppose that $[s, t]\subset I$ and $R\geq R_{\min}$ on $M\times [s, t]$. Let $(z, s)\in M\times I$ be an $H_n$-center of $(x, t)\in M\times I$. Then for any $\varepsilon>0$ and $y\in M$, we have
\begin{equation}\label{Equ: Heat kernel upper bound 1}
K(x, t; y,s)\leq \frac{C(R_{\min}(t-s),\varepsilon)}{(t-s)^{n/2}}\exp\left(-\cN_{x,t}(t-s)\right)\exp\left(-\frac{d^2_s(z,y)}{C(\varepsilon)(t-s)}\right).
\end{equation}
\end{lemma}

For more results on the entropy and heat kernel bounds on Ricci flow, we refer the readers to \cite{Bam20a}.


\section{Proof of the improved no-local-collapsing theorem}

In this section, we come to prove Theorem \ref{nlc2}. First, we need the following proposition.
\begin{proposition}\label{Prop: estimate of reduced length}
For any $T>0$, consider a smooth Ricci flow $(M,(g(t))_{t\in (-T,0)})$ on compact $n$-dimensional manifold, and choose $(s,t)\subset (-T,0)$ with $s+T>\varepsilon>0$. Consider a $C^1$ spacetime curve $\gamma:(0, t-s)\to M\times (-T, 0)$ with $\gamma (\tau)\in M\times \left\{t-\tau\right\}$ between points $\gamma(0)=x$ and $\gamma(t-s)=y$. Then we have
\begin{equation}\label{Equ: estimate of reduced length by nash entropy}
d^{g(s)}_{W_1}\left(\delta_{y,s}, \nu_{x,t;s}\right)\leq C(\varepsilon)\left(1+\frac{\mathcal{L}(\gamma)}{2\sqrt{t-s}}-\cN_{x,t}(t-s)\right)^{\frac{1}{2}}\sqrt{t-s}.
\end{equation}
\end{proposition}
\begin{proof}
This is obtain by Bamler in \cite[Lemma 21.2]{Bam20c}, we include the proof for the convenience of the readers.

First, from the definition of the reduced distance, we have
$$\ell_{(x,t)}(y,s)\leq \frac{\mathcal{L}(\gamma)}{2\sqrt{t-s}}.$$
Hence from the estimate of Perelman \cite[Corollary 9.5]{Per1}, we have the following lower bound on the heat kernel
\begin{equation}\label{Equ: Heat kernel lower bound}
\begin{split}
K(x, t; y,s)\geq & \frac{1}{(4\pi(t-s))^{n/2}}\exp\left(-\ell_{(x,t)}(y,s)\right)\\
\geq& \frac{1}{(4\pi(t-s))^{n/2}}\exp\left(-\frac{\mathcal{L}(\gamma)}{2\sqrt{t-s}}\right).\\
\end{split}
\end{equation}
But, since $s+T>\varepsilon$, we can apply the heat kernel upper bound estimate of Bamler, say Lemma \ref{hkub}, to obtain
\begin{equation}\label{Equ: Heat kernel upper bound 2}
K(x, t; y,s)\leq \frac{C(\varepsilon)}{(t-s)^{n/2}}\exp\left(-\cN_{x,t}(t-s)\right)\exp\left(-\frac{d^2_s(z,y)}{C(\varepsilon)(t-s)}\right).
\end{equation}
where $(z, s)$ is an $H_n$-center of $(x, t)$. Hence by the triangle inequality and (\ref{Equ: relation of W1 and variance}), we have
\begin{equation}\label{Equ: estimate of W_1 distance 1}
\begin{split}
&\left|d^{g(s)}_{W_1}\left(\delta_{y,s}, \nu_{x,t;s}\right)-d_s\left(y, z\right)\right|\\
=&\left|d^{g(s)}_{W_1}\left(\delta_{y,s}, \nu_{x,t;s}\right)-d^{g(s)}_{W_1}\left(\delta_{y,s}, \delta_{z,s}\right)\right|\\
\leq& d^{g(s)}_{W_1}\left(\delta_{z,s}, \nu_{x,t;s}\right)\leq \sqrt{\var\left(\delta_{z,s}, \nu_{x,t;s}\right)}\\
\leq& \sqrt{H_n(t-s)}.\\
\end{split}
\end{equation}
Hence, using the trivial inequality $(a-b)^2\geq \frac{a^2}{2}-4b^2$ for all $a,b\in \mathbb{R}$, we obtain
$$H_n(t-s)\geq \frac{1}{2}d^{g(s)}_{W_1}\left(\delta_{y,s}, \nu_{x,t;s}\right)^2-4d^2_s\left(y, z\right).$$
Plugging this into inequality (\ref{Equ: Heat kernel upper bound 2}), we obtain
\begin{equation}\label{Equ: Heat kernel upper bound 3}
\begin{split}
K(x, t; y,s)\leq &\frac{C(\varepsilon)}{(t-s)^{n/2}}\exp\left(-\cN_{x,t}(t-s)\right)\exp\left(-\frac{d^{g(s)}_{W_1}\left(\delta_{y,s}, \nu_{x,t;s}\right)^2-C(n)(t-s)}{C(\varepsilon)(t-s)}\right)\\
\leq &\frac{C(\varepsilon)}{(t-s)^{n/2}}\exp\left(-\cN_{x,t}(t-s)-\frac{d^{g(s)}_{W_1}\left(\delta_{y,s}, \nu_{x,t;s}\right)^2}{C(\varepsilon)(t-s)}\right).\\
\end{split}
\end{equation}
Combining (\ref{Equ: Heat kernel upper bound 3}) with (\ref{Equ: Heat kernel lower bound}),we obtain
$$\frac{1}{(4\pi(t-s))^{n/2}}\exp\left(-\frac{\mathcal{L}(\gamma)}{2\sqrt{t-s}}\right)\leq \frac{C(\varepsilon)}{(t-s)^{n/2}}\exp\left(-\cN_{x,t}(t-s)-\frac{d^{g(s)}_{W_1}\left(\delta_{y,s}, \nu_{x,t;s}\right)^2}{C(\varepsilon)(t-s)}\right).$$
Rearranging this inequality, we obtain
$$d^{g(s)}_{W_1}\left(\delta_{y,s}, \nu_{x,t;s}\right)\leq C(\varepsilon)\left(1+\frac{\mathcal{L}(\gamma)}{2\sqrt{t-s}}-\cN_{x,t}(t-s)\right)^{\frac{1}{2}}\sqrt{t-s},$$
which finishes the proof of the proposition.
\end{proof}

Now we can prove Theorem \ref{nlc2}.
\begin{proof}[Proof of Theorem \ref{nlc2}]

After parabolic rescaling, we may assume without loss of generality that $r_0=1$, and then $r\leq 1$.

Let $\gamma: [0, 1]\to M\times [-1, 0]$ be the spacetime curve defined by:
$$\gamma(\tau)=(x_0, -\tau),~~~ \tau\in[0, 1],$$
which is what we usually called the worldline. Then from our assumption, we have
$$R(\gamma(\tau), -\tau)\leq \frac{A}{\tau},~~ \forall \tau\in (0, 1],$$
and hence we can compute
\begin{equation}\label{Equ: estimate of reduced length 1}
\begin{split}
\mathcal{L}(\gamma)=& \int_{0}^{1}\sqrt{\tau}\left(|\gamma'(\tau)|^2_{-\tau}+R(\gamma(\tau), -\tau)\right)d\tau\\
=& \int_{0}^{1}\sqrt{\tau}R(\gamma(\tau), -\tau)d\tau\\
\leq& \int_{0}^{1}\sqrt{\tau}\cdot \frac{A}{\tau}d\tau\\
\leq& 2A.
\end{split}
\end{equation}

Next, by the monotonicity of the Nash entropy, we have that $0\geq \cN_{x_0,0}(1)\geq \cN_{x_0,0}(2)$, and hence $0\leq -\cN_{x_0,0}(1)\leq -\cN_{x_0,0}(2)$. Hence, we can apply Proposition \ref{Prop: estimate of reduced length} together with (\ref{Equ: estimate of reduced length 1}) to obtain
\begin{equation}\label{Equ: estimate of W_1 distance 2}
\begin{split}
&d^{g(-1)}_{W_1}\left(\delta_{x_0,-1}, \nu_{x_0,0;-1}\right)\\
\leq& C\left(1+\frac{\mathcal{L}(\gamma)}{2\sqrt{0-(-1)}}-\cN_{x_0, 0}(1)\right)^{\frac{1}{2}}\sqrt{0-(-1)}\\
\leq& C\left(1-\cN_{x_0, 0}(1)\right)^{\frac{1}{2}}.\\
\end{split}
\end{equation}
where $C=C(n, A)$. Hence, by applying the Harnack estimate of Nash entropy proved by Bamler, say Lemma \ref{grad}, we have
\begin{equation}\label{Equ: estimate of Nash entropy 1}
\begin{split}
&\left|\cN^*_{-2}(x_0, -1)-\cN^*_{-2}(x_0, 0)\right|\\
\leq& \left(\frac{n}{2((-1)-(-2))}+n\right)^{\frac{1}{2}}d^{g(-1)}_{W_1}\left(\delta_{x_0,-1}, \nu_{x_0,0;-1}\right)+\frac{n}{2}\log\left(\frac{0-(-2)}{(-1)-(-2)}\right)\\
\leq& C \left(C-\cN^*_{-2}(x_0, 0)\right)^{\frac{1}{2}}+C\\
\leq& 100C^2+\frac{1}{100}\left(C-\cN^*_{-2}(x_0, 0)\right)+C.
\end{split}
\end{equation}
Hence we obtain
\begin{equation}\label{Equ: estimate of Nash entropy 2}
\frac{99}{100}\cN_{-2}^*(x_0, 0)\geq \cN_{-2}^*(x_0, -1)-C
\end{equation}
But by the volume non-inflating estimate, say Lemma \ref{noninfl}, since we can assumed without loss of generality that $R\geq -n$, we have
$$A^{-1}\leq \vol_{g(-1)}(B_{g(-1)}(x_0, 1))\leq C\exp\left(\cN_{x_0,-1}(1)\right),$$
which can be written as 
$$\cN_{-2}^*(x_0, -1)=\cN_{x_0, -1}(1)\geq -C(n, A).$$
Combining this with (\ref{Equ: estimate of Nash entropy 2}), we obtain 
\begin{equation}\label{Equ: estimate of Nash entropy 3}
\cN_{-2}^*(x_0, 0)\geq -C,
\end{equation}
for some $C=C(n, A)\leq \infty$.

Next, by the gradient estimate of the Nash entropy in Lemma \ref{grad} , we have
\begin{equation}\label{Equ: estimate of Nash entropy 4}
\left|\cN^*_{-2}(x_0, 0)-\cN^*_{-2}(x, 0)\right|\leq C(n)d_0(x_0, x)\leq C(n)A,
\end{equation}
Combining (\ref{Equ: estimate of Nash entropy 3}) and (\ref{Equ: estimate of Nash entropy 4}) we obtain
$$\cN_{x, 0}(2)=\cN^*_{-2}(x, 0)\geq -C,$$
Then we apply the monotonicity of the Nash entropy to obtain (note that $r\leq 1$)
$$\cN_{x, 0}(r^2)\geq \cN_{x, 0}(2)\geq -C.$$
But we have assumed that $R\leq r^{-2}$ on $B_{g(0)}(x, r)$, we can apply the lower volume bound estimate, say Lemma \ref{noncol}, to obtain that:
$$\vol_{g(0)}(B_{g(0)}(x, r))\geq c(n)\exp\left(\cN_{x, 0}(r^2)\right)r^n\geq \kappa(n, A)r^n,$$
which finishes the proof of the theorem.
\end{proof}


\end{document}